\documentclass[a4paper,11pt]{amsart}

\usepackage{amsmath,amssymb}%
\usepackage{exscale}%
\usepackage{amsfonts}%
\usepackage{palatino}
\usepackage{amsthm}
\usepackage{epsfig}
\usepackage{hyperref}
\usepackage{multicol}
\usepackage[english]{babel}
\usepackage{mathtools}
\usepackage{cite}
\newcommand{\inner}[3][]{#1\langle #2,#3 #1\rangle}

\let\epsilon\varepsilon

\let\phi\varphi

\let\theta\vartheta

\DeclareMathOperator{\dd}{d\!}

\providecommand{\abs}[2][]{#1\lvert#2#1\rvert}
\providecommand{\norm}[2][]{#1\lVert#2#1\rVert}

\providecommand{\inv}{^{-1}}

\usepackage[T1]{fontenc}
\usepackage[latin1]{inputenc}
\usepackage{paralist}
\usepackage{enumitem,calc}
\usepackage{multicol}
\usepackage{wrapfig}
\newtheorem{theorem}{Theorem}[section]
\newtheorem{proposition}[theorem]{Proposition}
\newtheorem{corollary}[theorem]{Corollary} 
\theoremstyle{definition}
\newtheorem{definition}[theorem]{Definition}
\newtheorem{example}[theorem]{Example}
\newtheorem{remark}[theorem]{Remark}
\newtheorem{conjecture}[theorem]{Conjecture}

\DeclareMathOperator{\tr}{tr}

\newcommand{\R}{\mathbb{R}}
\newcommand{\C}{\mathbb{C}}

\newcommand{\T}{\mathbb{T}}

\newcommand{\cS}{\mathcal{S}}

\begin{document}
\pagestyle{plain} 
\title{Projection operators on matrix weighted $L^p$
  and a simple sufficient Muckenhoupt condition} \author{Morten
  Nielsen and Morten Grud Rasmussen}  \date{\today}
\begin{abstract}
  Boundedness for a class of projection operators, which includes the
  coordinate projections, on matrix weighted $L^p$-spaces is
  completely characterised in terms of simple scalar conditions.
  Using the projection result, sufficient conditions, which are
  straightforward to verify, are obtained that ensure that a given
  matrix weight is contained in the Muckenhoupt matrix $A_p$
  class. Applications to singular integral operators with product
  kernels are considered.
\end{abstract}
\thanks{Supported by the Danish Council for Independent Research | Natural
Sciences, grant 12-124675, "Mathematical and Statistical Analysis of
Spatial Data".}

\maketitle

\section{Introduction}
Singular integral operators form a natural generalisation of the
classical Hilbert transform, and the action of such operators on
$L^p(\R)$ has been studied in great detail. The theory was extended in
the '70's to include weighted $L^p$-spaces, with the seminal
contribution being the paper by Hunt, Wheeden, and Muckenhoupt
\cite{MR0312139}, where the Hilbert transform is shown to be bounded
on weighted $L^p$, $1<p<\infty$, if and only if the weight satisfy the
so-called Muckenhoupt $A_p$-condition. Even though the
$A_p$-conditions are quite involved, the Hunt, Wheeden, and
Muckenhoupt results are still very much operational since quite large
classes of, e.g., polynomial weights are known to satisfy the
respective conditions, see \cite{MR890662}.

A further generalisation of the Hilbert transform result to a vector
valued setup is straightforward in the non-weighted case, but it posed
a long-standing challenge to find a suitable generalisation in the
(matrix-)weigh\-ted setup. A breakthrough came with the results
\cite{MR1428818,MR1478786} of Treil and Volberg for $p=2$. This lead
to a correct definition of matrix $A_p$ weights for $1<p<\infty$, see
\cite{MR1423034}. The reader may consult \cite{MR2737763} for an
application of the Treil-Volberg result and the $A_2$-matrix condition
to applied harmonic analysis.

The matrix $A_p$ condition is considerably more complicated than the
scalar condition, and there are no known straightforward sufficient
conditions on a matrix weight to ensure membership in the $A_p$ class
except in very special cases (e.g., for diagonal weights and for
weights with strong pointwise bounds on its spectrum).  Bloom
\cite{MR805955,MR996354} has considered sufficient conditions for the
matrix $A_2$-condition in terms of certain weighted BMO-spaces.

In the present paper, we study and characterise a family of projection
operators on matrix weighted $L^p$. The family contains the coordinate
projections as special cases. The characterisation is given in terms
of simple scalar conditions. We then apply the projection result in
conjunction with the Treil-Volberg characterisation of matrix
$A_p$-weights to obtain a simple sufficient condition for a matrix to
satisfy the $A_p$-condition. We show that the new sufficient condition
covers many known examples of non-trivial matrix $A_p$ weights, such as
the ones considered by Bownik in \cite{MR1857041}. However, we do
provide an example of an $A_2$ matrix weight violating our condition
so the condition is {\em not} exhaustive.

As an application of the theory, we consider a family of singular
integral operators with product type kernels in the matrix weighted
setup. 

\section{Muckenhoupt matrix weights}
Let us first give a brief review of the $A_p$ condition following
\cite{MR1423034}.  We consider a domain $D\in\{\R^d, \R^m\times \R^n,
\T^d, \T^m\times \T^n\}$, where $\T^d$ denotes the $d$-dimensional
torus, and an associated measurable map $W\colon D\rightarrow
\C^{N\times N}$, with values in the non-negative definite matrices.

We introduce a family $\mathcal{S}_D$ of subsets of $D$. For
$D\in\{\R^d, \T^d\}$, $\mathcal{S}_D$ is the collection of all
Euclidean balls in $D$, while in the product case, i.e.\ when $D\in\{
\R^m\times \R^n, \T^m\times \T^n\}$, $\mathcal{S}_D$ is the collection
of all product sets $B_r\times B_{r'}$ where $B_r$ is a ball in $
\R^m$ [$ \T^m$] and $B_{r'}$ is a ball in $ \R^n$ [$\T^n$].  We define
the following family of metrics:
\begin{equation*}
  \rho_t(x)=\|W^{1/p}(t)x\|,\qquad
  x\in\C^N,\quad t\in D,
\end{equation*}
with the dual metric given by
\begin{equation*}
  \rho_t^*(x):=\sup_{y\not= 0} \frac{|\langle
    x,y\rangle|}{\rho_t(y)}=
  \|W^{-1/p}(t)x\|, \qquad x\in\C^N,\quad t\in D.
\end{equation*}
We now average $\rho_t$ over $E\in \mathcal{S}_D$
\begin{equation*}
  \rho_{p,E}(x):=\left (\frac1{|E|}\int_E
    [\rho_t(x)]^p\,\dd t\right)^{1/p},
\end{equation*}
and likewise for the dual metric
\begin{equation*}
  \rho^*_{p,E}(x):=\left (\frac1{|E|}\int_E [\rho^*_t(x)]^{q}\,\dd t\right)^{1/q},
\end{equation*}
with $q$ being $p$'s H\"older conjugate, $1=p\inv+q\inv$.

The $A_p$ condition can then be stated as follows.
\begin{definition}
  For $1<p<\infty$, we say that $W$ is an $A_p(N,D,\cS_D)$
  matrix weight if $W\colon D\rightarrow \C^{N\times N}$ is measurable
  and positive definite a.e.\ such that $W$ and $W^{-q/p}$ are locally
  integrable and there exists $C<\infty$ such that
  \begin{equation}\label{eq:Ap}
    \rho_{q,E}^*\leq C(\rho_{p,E})^*,\qquad E\in\mathcal{S}_D.
  \end{equation}
\end{definition}
\begin{remark}\label{rem:a}
  Notice that $\rho_t^*(x)=\|(W^{-q/p})^{1/q}(t)x\|$ so  
  \begin{equation}\label{eq:dual}
    W\in A_p(N,D,\cS_D) \textup{ if and only if } W^{-q/p}\in A_{q}(N,D,\mathcal{S}_D).
  \end{equation}
\end{remark}
In the following, we will sometimes relax the notation
$A_p(N,D,\cS_D)$ by leaving out the $N$, the $D$, and/or the $\cS_D$
if their values are clear from the context. Note that $A_p(1)$ is
simply the set of scalar Muckenhoupt weights.

Roudenko introduced an equivalent condition to \eqref{eq:Ap} in
\cite{MR1928089} which is often more straightforward to verify. In
fact, Roudenko only considered the case $D=\R^d$, but the reader can
easily verify that her proof in \cite{MR1928089} works verbatim in the
product and/or torus setup too. Condition \eqref{eq:Ap} holds if and
only if $W\colon D\rightarrow \C^{N\times N}$ is measurable and positive
definite a.e.\ such that $W$ and $W^{-q/p}$ are locally integrable and
there exists $C'<\infty$ such that
\begin{equation}\label{eq:Roudenko}
  \int_E\left( \int_E \big\|W^{1/p}(x)W^{-1/p}(t)\big\|^{q} \frac{\dd t}{|E|}\right)^{p/q} \frac{\dd x}{|E|}\leq C',\qquad E\in \mathcal{S}_D.
\end{equation}
  
For scalar weights defined on $\R^m\times \R^n$, it is well-known that
a product Muckenhoupt condition implies a uniform Muckenhoupt
condition in each variable, see \cite{MR766959}. Condition \eqref{eq:Roudenko} can be used to prove a
similar result for product matrix weights. We have the following
result.
  
\begin{proposition}\label{prop:Rou}
  Suppose $W\in A_p(\R^m\times \R^n)$. Then the weight $x\mapsto
  W(x,y)$, obtained by fixing the variable $y\in \R^n$ is uniformly in
  $A_p(\R^m)$ for a.e.\ $y\in \R^n$.
\end{proposition}
\begin{proof}
  Given a ball $B\subset \R^m$, we let
  $B_\epsilon:=B_\epsilon(y)\subset \R^n$ be the ball of radius
  $\epsilon$ about $y\in \R^n$.  First suppose $p\leq q$. Since $W\in
  A_p(\R^m\times \R^n)$ there exists a constant $c_W$ independent of
  $B\times B_\epsilon$ such that
  \begin{align*}
    \frac{1}{|B_\epsilon|^2}&\int_{B_\epsilon}\int_{B_\epsilon}\bigg[
    \int_B\bigg(\int_B \|{W}^{1/p}(x,y)W^{-1/p}(x',y')\|^{q}
    \frac{\dd x'}{|B|}  \bigg)^{p/q}\frac{\dd x}{|B|}\bigg] \dd y\dd y'\\
    &\leq \int_{B_\epsilon} \int_B\bigg( \int_{B_\epsilon}\int_B
    \|{W}^{1/p}(x,y)W^{-1/p}(x',y')\|^{q}
    \frac{\dd x'\,\dd y'}{|B|\cdot |B_\epsilon|}  \bigg)^{p/q}\frac{\dd x\,\dd y}{|B|\cdot |B_\epsilon|}  \\
    &\leq c_W,
  \end{align*}
  where we have used H\"older's inequality.  Now we let
  $\epsilon\rightarrow 0$ and use Le\-bes\-gue's differentiation theorem
  to conclude that for almost every $y\in\R^n$,
  \begin{align*}
    \int_B\bigg(\int_B \|{W}^{1/p}(x,y){W}^{-1/p}(x',y)\|^{q}
    \frac{\dd x'}{|B|} \bigg)^{p/q}\frac{\dd x}{|B|}&\leq c_W.
  \end{align*}
  Since $c_W$ is independent of $B$, it follows that $x\mapsto
  W(x,y)$ is uniformly in ${A}_p(\R^m)$ for a.e.\ $y\in\R^n$.
  In the case $q<p$, we use that $W^{-q/p}\in {A}_{q}(\R^m\times
  \R^n)$ by \eqref{eq:dual}, which implies the following estimate
  \begin{equation*}
    \int_{B_\epsilon}
    \int_B\bigg( \int_{B_\epsilon}\int_B \|{W}^{1/p}(x,y)W^{-1/p}(x',y')\|^{p}
    \frac{\dd x'\,\dd y'}{|B|\cdot |B_\epsilon|}  \bigg)^{q/p}\frac{\dd x\,\dd y}{|B|\cdot |B_\epsilon|}  \leq c_W.
  \end{equation*}  
  By repeating the argument from the $p\leq q$ case, we conclude that the map
  $x\mapsto {W}^{-q/p}(x,y)$ is in ${A}_{q}(\R^m)$ for a.e.\
  $y\in\R^n$ which again by \eqref{eq:dual} implies that $x\mapsto
  {W}(x,y)$ is in $A_p(\R^m)$ for a.e.\ $y\in\R^n$.
\end{proof}
 
A similar result clearly holds true for the weight $y\mapsto
W(x,y)$. The periodic case, i.e. the case $W\in A_p(\T^m\times \T^n)$,
is also similar.

\section{Projection operators}
Recall that at scalar weight is a measurable function which is
positive a.e. If $w\colon D\to\C$ is a scalar weight, we define the
weighted space $L^p(w)$ as the set of measurable functions $f\colon D\to\C$
for which
\begin{equation*}
  \norm{f}_{L^p(w)}:=\Biggl(\int_D\abs{f}^pw\dd\mu\Biggr)^{1/p}
\end{equation*}
is finite, where $\mu$ is the measure on $D$. Likewise, if $W\colon
D\to\C^{N\times N}$ is a matrix-valued function which is measurable
and positive definite a.e., then the space $L^p(W)$ is the set of
measurable functions $f\colon D\to\C^N$ with 
\begin{equation*}
  \norm{f}_{L^p(W)}:=\Biggl(\int_D\abs{W^{1/p}f}^p\dd\mu\Biggr)^{1/p}<\infty.
\end{equation*}
Obviously, in order to turn  $L^p(w)$ and $L^p(W)$ into Banach spaces, one has to factorize over 
$\{f\colon D\to\C;\norm{f}_{L^p(w)}=0\}$ and $\{f\colon D\to\C^N;\norm{f}_{L^p(W)}=0\}$, respectively.
We can now state our main result giving a full characterization of a certain class of projections from $L^p(W)$ to $L^p(w)$.
\begin{theorem}\label{th:projection}
  Let $W=(w_{ij})_{i,j=1}^N\colon D\to\C^{N\times N}$ be a
  matrix-valued function which is measurable and positive definite
  a.e., $w\colon D\to\C$ be a scalar weight and $r\colon D\to \C^{N}$
  be a unit vector valued function. Then the projection in the
  direction of $r$, $P_r\colon L^p(W)\to L^p(w)$ given by
  $P_r(f)=\inner{f}{r}$ is bounded if and only if
  \begin{equation}
    w^{\frac{1}{p}}\norm{W^{-\frac{1}{p}}r}\in L^\infty\label{eq:projectionthm1}.
  \end{equation}
  In particular, if we denote the entries of powers of $W$ by
  $W^{s}=(w_{ij}^{(s)})$, where $s$ is any real number, then
  $P_k=P_{e_k}\colon L^p(W)\to L^p(w_{kk})$ is bounded if and only if
  \begin{equation}
w_{kk}^{\frac{2}{p}}{w_{kk}^{(-\frac{2}{p})}}\in L^\infty,\label{eq:projectionthm2}
\end{equation}
and if $\lambda_i\colon D\to\R$ and $v_i\colon D\to C^N$ are
eigenvalues and eigenvectors, respectively, of $W$, then the
projection $P_{v_i}\colon L^p(W)\to L^p(\lambda_i)$ is always bounded.
\end{theorem}
\begin{proof}
  We begin with the necessity part. Assume therefore that $P_r$ is
  bound\-ed. Then there exists a constant $C$ such that
  $C\norm{f}_{L^p(W)}^p\ge \norm{P_rf}_{L^p(w)}^p$. Now let
  $\{f_\epsilon\}_{\epsilon>0}$ , $f_\epsilon\colon D\to\C$ be an
  approximate identity and $T_k$ the translation operator,
  $T_kf=f(\cdot-k)$. Then
  \begin{align*}
    C&=C\norm[\Big]{T_kf_\epsilon^{\frac{1}{p}}\tfrac{W^{-\frac{1}{p}}r}{\norm{W^{-\frac{1}{p}}r}}}_p^p=C\norm[\Big]{T_kf_\epsilon^{\frac{1}{p}}\tfrac{W^{-\frac{2}{p}}r}{\norm{W^{-\frac{1}{p}}r}}}_{L^p(W)}^p\\
    &\ge \norm[\Big]{T_kf_\epsilon^{\frac{1}{p}}\inner[\big]{\tfrac{W^{-\frac{2}{p}}r}{\norm{W^{-\frac{1}{p}}r}}}{r}}_{L^p(w)}^p=\norm[\big]{T_kf_\epsilon^{\frac{1}{p}}\norm{W^{-\frac{1}{p}}r}}_{L^p(w)}^p\\
    &=\norm[\big]{T_kf^{\frac{1}{p}}_\epsilon w^{\frac{1}{p}}\norm{W^{-\frac{1}{p}}r}}_p^p.
  \end{align*}
  Letting $\epsilon\to0$ we get that
  $w^{\frac{1}{p}}(k)\norm{W^{-\frac{1}{p}}(k)r(k)}\le C^{\frac{1}{p}}$ for a.e.\
  $k$. 

  We now show that essential boundedness of
  $w^{\frac{1}{p}}\norm{W^{-\frac{1}{p}}r}$ implies boundedness of
  $P_r$. Assume therefore that
  $w^{\frac{1}{p}}\norm{W^{-\frac{1}{p}}r}\le C$ a.e. Write $q$ for
  the H\"older conjugate of $p$. Then for every $\psi\in L^q(w)$:
  \begin{align*}
    C\norm{f}_{L^p(W)}\norm{\psi}_{L^q(w)}&\ge\norm{f}_{L^p(W)}\Bigl(\int\abs{\psi w^{\frac{1}{q}}}^q\bigl(\abs{w}^{\frac{1}{p}}\norm{W^{-\frac{1}{p}}r}\bigr)^q\dd\mu\Bigr)^{\frac{1}{q}}\\
    &=\norm{f}_{L^p(W)}\Bigl(\int\abs{\psi w}^q\norm{W^{-\frac{1}{p}}r}^q\dd\mu\Bigr)^{\frac{1}{q}}\\
    &\ge\int\norm{W^{\frac{1}{p}}f}\norm{W^{-\frac{1}{p}}r\psi w}\dd\mu\\
    &\ge\int\abs{\inner{W^{\frac{1}{p}}f}{W^{-\frac{1}{p}}r\psi w}}\dd\mu\\
    &=\int\abs{P_rf\psi w}\dd\mu,
  \end{align*}
  so $P_r\colon L^p(W)\to L^p(w)$ is bounded.

  Note that \eqref{eq:projectionthm1} reduces to
  \eqref{eq:projectionthm2} when $r\equiv e_k$, the constant function
  with all coordinates except the $k$'th being zero, and
  $w=w_{kk}$. Indeed,
  \begin{equation*}
    \norm{W^{-\frac{1}{p}}e_k}^2=\inner{W^{-\frac{1}{p}}e_k}{W^{-\frac{1}{p}}e_k}=\inner{e_k}{W^{-\frac{2}{p}}e_k}=w_{kk}^{(-\frac{2}{p})}.
  \end{equation*}
  We finish the proof of the theorem by noting that 
  \begin{equation*}
    \lambda_i^{\frac{1}{p}}\norm{W^{-\frac{1}{p}}v_i}=1\quad\textup{a.e.},
  \end{equation*}
  which is clearly in $L^\infty$.
\end{proof}

\section{A sufficient matrix Muckenhoupt condition}

Here we consider an application of the projection result to derive
operational sufficient conditions for a matrix weight to be in the
Muckenhoupt $A_p$ class. The matrix $A_p$ condition introduced in
\cite{MR1423034} is rather involved and it may be difficult to verify
for a given matrix. The simpler $A_2$-case was settled in \cite{MR1428818,MR1478786}. An additional advantage of the projection approach
is that it applies to both the regular matrix $A_p$ condition and to
the corresponding product setup.

We will need the following  fundamental characterization of the matrix
condition $A_p$, see \cite{MR2015733}. Recall that the Riesz transform $R_j\colon L^p(\R^d)\to L^p(\R^d)$ is given by
$$\mathcal{F}(R_j f)(\xi)=i\,\frac{\xi_j}{\xi}\mathcal{F}f(\xi),\qquad j=1,2,\dotsc, d.$$
\begin{theorem}\label{th:Treil}
  Let $W=(w_{ij})\colon \R^d\to\C^{N\times N}$ be a matrix-valued
  function which is measurable and positive definite a.e. Then $W\in
  A_p(N,\R^d)$, $1<p<\infty$, if and only if the Riesz transforms $R_j\colon
  L^p(W)\rightarrow L^p(W)$ are bounded for all $j=1,2,\dotsc,d$.
\end{theorem}

We now consider the product setup. For simplicity we focus on the case $D=\R^m\times\R^n$. The case $D=\T^m\times\T^n$ can be treated in a similar fashion. We write $z=(x,y)\in D$ with $x\in \R^m$ and $y\in\R^n$. Let 
$\tilde{R}_{i}^x$ denote the operator $R_i^x\otimes \text{Id}_y$, where $R_i^x$ is the Riesz transform acting on $\R^m$.  Similarly, we let $\tilde{R}_{j}^y$ denote the operator $\text{Id}_x\otimes  R_j^y$, where $R_j^y$ is the Riesz transform acting on $\R^n$.  We have the following Corollary to Theorem \ref{th:Treil}.

\begin{corollary}\label{cor:treil}
Let $D=\R^m\times\R^n$ and let $W=(w_{ij})\colon D \to\C^{N\times N}$ be a matrix-valued
  function which is measurable and positive definite a.e. Then $W\in
  A_p(N,D,\cS_D)$, $1<p<\infty$, if and only if $\tilde{R}_{i}^x,\tilde{R}_j^y\colon
  L^p(W)\rightarrow L^p(W)$ are bounded for $i=1,2,\dotsc,m$ and $j=1,2,\dotsc,n$.
\end{corollary}

\begin{proof}
Suppose $A_p(N,D,\cS_D)$. Then  Proposition \ref{prop:Rou} shows that $W(x,y)$ is uniformly  $A_p$ in each variable separately (a.e.). We can then use Theorem~\ref{th:Treil} together with a simple iteration argument to deduce that the operators $\tilde{R}_{i}^x,\tilde{R}_j^y\colon
  L^p(W)\rightarrow L^p(W)$ are bounded for $i=1,2,\dotsc,m$ and $j=1,2,\dotsc,n$. Conversely, suppose that $\tilde{R}_{i}^x,\tilde{R}_j^y\colon
  L^p(W)\rightarrow L^p(W)$ are bounded for all $i=1,2,\dotsc,m$ and $j=1,2,\dotsc,n$. 
  Take any $f=(f_i)_{i=1}^N$ with $f_i\in C^\infty_c(\R^n)$, $i=1,\dotsc, N$, and fix $x_0\in \R^m$.  Let $\phi_\epsilon\in C^\infty_c(\R^m)$ be an approximation to the identity centered at $x_0\in \R^m$. We let $R_jf$ denote the vector $(R_jf_i)_{i=1}^m$. Then using the boundedness of $\tilde{R}_{j}^y$,
\begin{align*}
\int_{\R^m}\int_{\R^n} &\phi_\epsilon(x)|W^{1/p}(x,y)R_jf(y)|^p \dd x \dd y\\
&\leq C
\int_{\R^m}\int_{\R^n} \phi_\epsilon(x)|W^{1/p}(x,y)f(y)|^p \dd x \dd y.
\end{align*}
We let $\epsilon\rightarrow 0$ to conclude that almost surely
$$\int_{\R^n} |W^{1/p}(x_0,y)R_jf(y)|^p  \dd y
\leq C \int_{\R^n} |W^{1/p}(x_0,y)f(y)|^p \dd y,\qquad x_0\in\R^m.$$
We now use Theorem \ref{th:Treil} to conclude that $y\rightarrow W(x,y)$ is uniformly in $A_p$ for a.e.\ $x$. A similar argument using $\tilde{R}_{i}^x$ shows that $x\rightarrow W(x,y)$ is uniformly in $A_p$ for a.e.\ $y$. Using these uniform bounds it follows easily that $W \in A_p(N,D,\cS_D)$.
\end{proof}

We can now give a sufficient condition for membership in $ A_p(N,D,\cS_D)$.

\begin{theorem}\label{th:An}
  Let $W=(w_{ij})\colon D\to\C^{N\times N}$ be a matrix weight which
  is invertible a.e. Fix $1<p<\infty$ and denote the entries of powers
  of $W$ by $W^{s}=(w_{ij}^{(s)})$, where $s$ is any real
  number. Suppose that
  \begin{equation*}
    w_{kk}^{(\frac{2}{p})}{w_{kk}^{(-\frac{2}{p})}}\in L^\infty,\qquad k=1,2,\dotsc,N
  \end{equation*}
  and that $(w_{kk}^{(\frac{2}{p})})^\frac{p}{2}\in A_p(1)$ for
  $k=1,2,\dotsc,N$. Then $W\in A_p(N,D,\cS_D)$.
\end{theorem}
\begin{proof}
  Take $f\in L^p(W)\cap C^\infty_c(D)$. We write $f=\sum_{j=1}^N f_j
  e_j=\sum_{j=1}^N P_j(f)e_j$, and note that by definition the
  vector-valued operators $\tilde{R}_{i}^x$ and $\tilde{R}_j^y$ act coor\-di\-nate-wise on $f$, so it
  follows that for  $K\in \{\tilde{R}_{i}^x\}_{i=1}^m\cup
  \{\tilde{R}_j^y\}_{j=1}^n$,
  \begin{equation*}
    Kf:= \sum_{j=1}^N (Kf_j) e_j.
  \end{equation*}
  According to Corollary~\ref{cor:treil}, the scalar-valued
  transform $K$ is bounded on
  $L^p((w_{kk}^{(\frac{2}{p})})^{\frac{p}{2}})$ for $k=1,2,\dotsc,N$, so
  we obtain
\begin{align*}
\|Kf\|_{L^p(W)}&\leq 
\sum_{j=1}^N \|(Kf_j) e_j\|_{L^p(W)}\\
&=\sum_{j=1}^N \|Kf_j\|_{L^p((w_{jj}^{(\frac{2}{p})})^\frac{p}{2})}\\
&\leq C\sum_{j=1}^N \|f_j\|_{L^p((w_{jj}^{(\frac{2}{p})})^\frac{p}{2})}\\
&\leq C'\|f\|_{L^p(W)},
\end{align*}
where we used that $P_j\colon L^p(W)\rightarrow
L^p((w_{jj}^{\frac{2}{p}})^\frac{p}{2})$ is bounded by
Theorem~\ref{th:projection}.  Now we use Corollary~\ref{cor:treil} to
conclude that $W\in  A_p(N,D,\cS_D)$.
\end{proof}
Note that Theorem~\ref{th:An} gives us an easily verifiable (at least
for $p=2$) sufficient condition for $W\in A_p(N)$. It is known that
$W\in A_p(N)$ implies that $(w_{kk}^{(\frac{2}{p})})^\frac{p}{2}\in
A_p(1)$ for $k=1,2,\dotsc, N$, indicating a possibility that the
conditions of Theorem \ref{th:An} in fact characterize $A_p(N)$.
However, this is not the case as the following example illustrates.
\begin{example}
  Let $W$ be given by
  \begin{equation*}
    [0,1]\ni x\mapsto W(x)=\begin{pmatrix}
      \sqrt{x}+\tfrac{1}{\sqrt{x}}&\tfrac{i}{\sqrt{x}}\\
      -\tfrac{i}{\sqrt{x}}&\tfrac{1}{\sqrt{x}}
    \end{pmatrix}.
  \end{equation*}
  Then $W\in A_2(2)$ but \eqref{eq:projectionthm2} with $p=2$ fails
  to hold. Indeed, $\det(W)\equiv1$, so
  \begin{equation*}
    w_{11}w_{11}^{(-1)}=w_{22}w_{22}^{(-1)}=1+\frac{1}{x}\not\in L^\infty,
  \end{equation*}
  but
  \begin{equation*}
    W_{a,b}=\int_a^bW(x)\dd x=
    \begin{pmatrix}
      \frac{2}{3}(b^{\frac{3}{2}}-a^{\frac{3}{2}})+2(\sqrt{b}-\sqrt{a})&2i(\sqrt{b}-\sqrt{a})\\
      -2i(\sqrt{b}-\sqrt{a})&2(\sqrt{b}-\sqrt{a})
    \end{pmatrix}
  \end{equation*}
and
\begin{equation*}
  W^{(-1)}_{a,b}=\int_a^bW\inv(x)\dd x=\begin{pmatrix}
      2(\sqrt{b}-\sqrt{a})&-2i(\sqrt{b}-\sqrt{a})\\
      2i(\sqrt{b}-\sqrt{a})&\frac{2}{3}(b^{\frac{3}{2}}-a^{\frac{3}{2}})+2(\sqrt{b}-\sqrt{a})
  \end{pmatrix}
\end{equation*}
so
\begin{equation*}
  W_{a,b}W^{(-1)}_{a,b}=\tfrac{4}{3}((b-a)^2-\sqrt{ab}(\sqrt{b}-\sqrt{a})^2)I_2
\end{equation*}
and hence
\begin{align*}\allowdisplaybreaks
  \norm[\big]{(\tfrac{1}{b-a}W_{a,b})^{\frac{1}{2}}(\tfrac{1}{b-a}W^{(-1)}_{a,b})^{\frac{1}{2}}}_F&=\tfrac{1}{b-a}\sqrt{\tr\bigl((W^{(-1)}_{a,b})^{\frac{1}{2}}W_{a,b}(W^{(-1)}_{a,b})^{\frac{1}{2}}\bigr)}\\
  &=\tfrac{1}{b-a}\sqrt{\tr\bigl(W_{a,b}W_{a,b}^{(-1)}\bigr)}\\
  &=\tfrac{1}{b-a}\sqrt{\tfrac{8}{3}(b-a)^2-\sqrt{ab}(\sqrt{b}-\sqrt{a})^2}\le\tfrac{2\sqrt{2}}{\sqrt{3}}, %\\
  % (b^2+a^2-\sqrt{ab}(a+b))&=(b-a)^2-\sqrt{ab}(\sqrt{b}-\sqrt{a})^2\\
  % \frac{\sqrt{x^2-\sqrt{h^2+xh}(\sqrt{x+h}-\sqrt{h})^2}}{x}&=\sqrt{1-\frac{1}{x^2}\sqrt{h^2+xh}(\sqrt{x+h}-\sqrt{h})^2}
  \end{align*}
  where the Frobenius norm was used for convenience.
\end{example}

\section{An application to vector valued singular integral operators}
Let us consider singular integral operators on the Euclidean product space
$\mathbb{R}^n\times \mathbb{R}^m$. Recall that a scalar weight
$w(x,y)$ satisfies the (product) Muckenhoupt $A_p(1,\mathbb{R}^n\times
\mathbb{R}^m)$-condition precisely when $w$ is uniformly in $A_p(1)$ for
each variable $x$ and $y$ separately. This makes it very easy to study
singular integral operators with a corresponding product structure on
$L^p(\mathbb{R}^n\times \mathbb{R}^m,w)$ using a simple iteration
argument.

For example, the product Hilbert transform
\begin{equation*}
  f\rightarrow \text{p.v.} \frac1 {xy} *  f
\end{equation*}
is bounded on $L^p(\mathbb{R}\times \mathbb{R},w)$,
$1<p<\infty$, whenever $w\in A_p(\mathbb{R}\times \mathbb{R}) $.

The case when the kernel is not separable but otherwise resemble a product Hilbert transform is much more complicated and has been studied in  e.g. \cite{MR664621}.

Suppose that $K$ is locally integrable on $\mathbb{R}^n\times \mathbb{R}^m$ away from the cross $\{x=0\}\cup
\{y=0\}$.

We let 
\begin{equation*}
  \Delta_h^1 K(x,y)=K(x+h,y)-K(x,y),
\end{equation*} 
\begin{equation*}
  \Delta_k^2 K(x,y)=K(x,y+k)-K(x,y),
\end{equation*} 
and 
\begin{equation*}
\Delta_{h,k}^{1,2} K(x,y)=\Delta_h^1(\Delta_k^2(K)).
\end{equation*}
The following 5 technical conditions for some $A<\infty$ and $\eta>0$
turn out to be important to establish boundedness of the operator
induced by $K$:

\begin{enumerate}
\item[(C.1)] $\big|\iint_{\alpha_1<|x|<\alpha_2,\beta_1<|y|<\beta_2} K(x,y) \dd x \dd y\big|\leq A$ for all $0<\alpha_1<\alpha_2$ and  $0<\beta_1<\beta_2$.
\item[(C.2)] For $K_1$ given by $K_1(x)=\int_{\beta_1<|y|<\beta_2}
  K(x,y) \dd y$ then $|K_1(x)|\leq A|x|^{-n}$ for all
  $0<\beta_1<\beta_2$, $|\Delta_h^1 K_1(x)|\leq
  A|h|^\eta|x|^{-n-\eta}$ for $|x|\geq 2|h|$, with a similar condition
  for $K_2(y)=\int_{\alpha_1<|x|<\alpha_2} K(x,y) \dd x$.
\item[(C.3)] $|K(x,y)|\leq A|x|^{-n}|y|^{-m}.$
\item[(C.4)] $|\Delta_h^1 K(x,y)|\leq A|h|^\eta|x|^{-n-\eta}|y|^{-m}$ if $|x|\geq 2|h|$ with a similar condition on
$\Delta_k^2 K(x,y)$.
\item[(C.5)] $|\Delta_{h,k}^{1,2} K(x,y)|\leq A (|h||k|)^\eta |x|^{-n-\eta}|y|^{-m-\eta}$ if $|x|\geq 2|h|$ and $|y|\geq 2|k|$. 
\end{enumerate}

\begin{theorem}[{\cite{MR664621}}]\label{th:stein}
  Let $1<p<\infty$. Suppose $K$ is locally integrable on
  $\mathbb{R}^n\times \mathbb{R}^m$ away from the cross $\{x=0\}\cup
  \{y=0\}$. Assume that $K$ satisfies $(C.1)-(C.5)$. Then the
  truncated kernels
  \begin{equation*}
    K_\epsilon^N(x,y)=K(x,y)\chi_{\epsilon_1<|x|<N_1}(x)\chi_{\epsilon_2<|y|<N_2}(y)
  \end{equation*}
  induce a uniformly bounded family of operators
  \begin{equation*}
    T_\epsilon^N(f):=f*K_\epsilon^N
  \end{equation*}
  on $L^p(\mathbb{R}^n\times \mathbb{R}^m,w)$ whenever $w\in
  A_p(1,\mathbb{R}^n\times \mathbb{R}^m)$. Morover,
  if 
  \begin{equation*}
    \smashoperator{\iint_{\alpha_1<|x|<\alpha_2, \beta_1<|y|<\beta_2}}
    K_\epsilon^N(x,y) \dd x \dd y,\quad \smashoperator{\int_{\alpha_1<|y|<\alpha_2}}
    K_\epsilon^N(x,y) \dd y, \quad\text{and} \quad
    \smashoperator{\int_{\beta_1<|x|<\beta_2}} K_\epsilon^N(x,y)
    \dd x
  \end{equation*}
  converge a.e.\ to limits as $\epsilon\rightarrow 0$ and
  $N\rightarrow \infty$ for all $0<\alpha_1<\alpha_2$,
  $0<\beta_1<\beta_2$, then for every $f\in L^p(\mathbb{R}^n\times
  \mathbb{R}^m,w)$, $T_K(f):=\lim_{\epsilon\rightarrow 0,N\rightarrow
    \infty} T_\epsilon^N(f)$ defines a bounded operator on
  $L^p(\mathbb{R}^n\times \mathbb{R}^m,w)$ for $w\in
  A_p(1,\mathbb{R}^n\times \mathbb{R}^m)$.
\end{theorem}

A natural extension of Theorem \ref{th:stein} would be to lift the
operator $T_K$ to the matrix-weighted case. There are at present some
technical obstacles that prevent us from carrying out this program in
full generality, but we can use the results in the previous sections
to obtain a partial result.

\begin{corollary}\label{cor:Tk}
  Let $1<p<\infty$ and let $K\colon\mathbb{R}^n\times
  \mathbb{R}^m\rightarrow \mathbb{C}$ be a kernel of the type
  considered in Theorem \ref{th:stein}. Suppose $W=(w_{ij})\colon
  \mathbb{R}^n\times \mathbb{R}^m\to\C^{N\times N}$ be a matrix weight
  which is invertible a.e. Then the operator $T_K$ lifted to the
  vector valued setting is bounded on $L^p(\mathbb{R}^n\times
  \mathbb{R}^m,W)$ provided
  that
  \begin{equation*}
    w_{kk}^{(\frac{2}{p})}{w_{kk}^{(-\frac{2}{p})}}\in
    L^\infty(\mathbb{R}^n\times \mathbb{R}^m),\qquad
    k=1,2,\dotsc,N
  \end{equation*}
  and that $(w_{kk}^{(\frac{2}{p})})^\frac{p}{2}\in
  A_p(1,\mathbb{R}^n\times \mathbb{R}^m)$ for $k=1,2,\dotsc,N$.
\end{corollary}
\begin{proof}
  Take $f\in L^p(\mathbb{R}^n\times \mathbb{R}^m,W)\cap
  C^\infty_c(\mathbb{R}^n\times \mathbb{R}^m)$, and write write the function as
  $f=\sum_{j=1}^N f_j e_j=\sum_{j=1}^N P_j(f)e_j$. It follows that
  \begin{equation*}
    T_Kf:= \sum_{j=1}^N (T_Kf_j) e_j,
  \end{equation*}
  so
  \begin{align*}
    \|T_Kf\|_{L^p(W)}&\leq \sum_{j=1}^N \|(T_Kf_j) e_j\|_{L^p(W)}=\sum_{j=1}^N \|T_Kf_j\|_{L^p((w_{jj}^{(\frac{2}{p})})^\frac{p}{2})}\\
    &\leq  C\sum_{j=1}^N \|f_j\|_{L^p((w_{jj}^{(\frac{2}{p})})^\frac{p}{2})}\\
    &\leq C'\|f\|_{L^p(W)},
  \end{align*} 
  where we used Theorem~\ref{th:stein}, and the projection result,
  Theorem \ref{th:projection}.
\end{proof}

\begin{conjecture}
  The conclusion of Corollary \ref{cor:Tk} holds true for any matrix
  weight $W=(w_{ij})\colon \mathbb{R}^n\times
  \mathbb{R}^m\to\C^{N\times N}$ in the set of Muckenhoupt weights
  $A_p(N,\mathbb{R}^n\times \mathbb{R}^m)$, $1<p<\infty$.
\end{conjecture}

\bibliographystyle{abbrv} 
\bibliography{projection}

\begin{thebibliography}{10}

\bibitem{MR805955}
S.~Bloom.
\newblock A commutator theorem and weighted {BMO}.
\newblock {\em Trans. Amer. Math. Soc.}, 292(1):103--122, 1985.

\bibitem{MR996354}
S.~Bloom.
\newblock Applications of commutator theory to weighted {BMO} and matrix
  analogs of {$A_2$}.
\newblock {\em Illinois J. Math.}, 33(3):464--487, 1989.

\bibitem{MR1857041}
M.~Bownik.
\newblock Inverse volume inequalities for matrix weights.
\newblock {\em Indiana Univ. Math. J.}, 50(1):383--410, 2001.

\bibitem{MR766959}
S.-Y.~A. Chang and R.~Fefferman.
\newblock Some recent developments in {F}ourier analysis and {$H^p$}-theory on
  product domains.
\newblock {\em Bull. Amer. Math. Soc. (N.S.)}, 12(1):1--43, 1985.

\bibitem{MR664621}
R.~Fefferman and E.~M. Stein.
\newblock Singular integrals on product spaces.
\newblock {\em Adv. in Math.}, 45(2):117--143, 1982.

\bibitem{MR2015733}
M.~Goldberg.
\newblock Matrix {$A_p$} weights via maximal functions.
\newblock {\em Pacific J. Math.}, 211(2):201--220, 2003.

\bibitem{MR0312139}
R.~Hunt, B.~Muckenhoupt, and R.~Wheeden.
\newblock Weighted norm inequalities for the conjugate function and {H}ilbert
  transform.
\newblock {\em Trans. Amer. Math. Soc.}, 176:227--251, 1973.

\bibitem{MR2737763}
M.~Nielsen.
\newblock On stability of finitely generated shift-invariant systems.
\newblock {\em J. Fourier Anal. Appl.}, 16(6):901--920, 2010.

\bibitem{MR890662}
F.~Ricci and E.~M. Stein.
\newblock Harmonic analysis on nilpotent groups and singular integrals. {I}.
  {O}scillatory integrals.
\newblock {\em J. Funct. Anal.}, 73(1):179--194, 1987.

\bibitem{MR1928089}
S.~Roudenko.
\newblock Matrix-weighted {B}esov spaces.
\newblock {\em Trans. Amer. Math. Soc.}, 355(1):273--314 (electronic), 2003.

\bibitem{MR1478786}
S.~Treil and A.~Volberg.
\newblock Continuous frame decomposition and a vector
  {H}unt-{M}uckenhoupt-{W}heeden theorem.
\newblock {\em Ark. Mat.}, 35(2):363--386, 1997.

\bibitem{MR1428818}
S.~Treil and A.~Volberg.
\newblock Wavelets and the angle between past and future.
\newblock {\em J. Funct. Anal.}, 143(2):269--308, 1997.

\bibitem{MR1423034}
A.~Volberg.
\newblock Matrix {$A_p$} weights via {$S$}-functions.
\newblock {\em J. Amer. Math. Soc.}, 10(2):445--466, 1997.

\end{thebibliography}
\end{document}